\documentclass[11pt, letterpaper,reqno]{amsart}

\usepackage{amsmath}
\usepackage{amssymb}
\usepackage{amsfonts}
\usepackage{ dsfont }
\usepackage{cases}
\usepackage{cite}
\usepackage[hyphens]{url}
\usepackage{hyperref}
\usepackage{amsthm}
\usepackage{upgreek}
\usepackage{color}

\numberwithin{equation}{section}
\DeclareMathOperator*{\esssup}{ess\,sup}
\DeclareMathOperator*{\essinf}{ess\,inf}

\newtheorem{theorem}{Theorem}[section]
\newtheorem{corollary}{Corollary}[section]
\newtheorem{lemma}{Lemma}[section]
\newtheorem{proposition}{Proposition}[section]
\newtheorem{definition}{Definition}[section]

\newtheorem{remark}{Remark}[section]

\newcommand{\leref}{Lemma~\ref}

\newcommand{\prref}{Proposition~\ref}
\newcommand{\thref}{Theorem~\ref}
\newcommand{\reref}{Remark~\ref}

\newcommand{\Rho}{{\pmb\rho}}
\newcommand{\Tau}{{\pmb\tau}}
\newcommand{\E}{\mathbb{E}}
\newcommand{\T}{\mathcal{T}}
\newcommand{\bT}{\mathbb{T}}
\newcommand{\eps}{\epsilon}
\newcommand{\cV}{\mathcal{V}}

\title[]{On a Stopping Game in continuous time}
\author[]{Erhan Bayraktar} \thanks{This research was supported in part by the National Science Foundation under grant DMS 0955463.}  
\address{Department of Mathematics, University of Michigan}
\email{erhan@umich.edu}
\author[]{Zhou Zhou}
\address{Department of Mathematics, University of Michigan}
\email{zhouzhou@umich.edu}
\date{\today}
\begin{document}
\maketitle
\begin{abstract}
On a filtered probability space $(\Omega,\mathcal{F},P,\mathbb{F}=(\mathcal{F}_t)_{0\leq t\leq T})$, we consider stopper-stopper games $\overline C:=\inf_{\Rho}\sup_{\tau\in\T}\E[U(\Rho(\tau),\tau)]$ and $\underline C:=\sup_{\Tau}\inf_{\rho\in\T}\E[U(\rho,\Tau(\rho))]$ in continuous time, where $U(s,t)$ is $\mathcal{F}_{s\vee t}$-measurable (this is the new feature of our stopping game), $\T$ is the set of stopping times, and $\Rho,\Tau:\T\mapsto\T$ satisfy certain non-anticipativity conditions. We show that $\overline C=\underline C$, by converting these problems into a corresponding Dynkin game. 
\end{abstract}

\section{Introduction}
On a filtered probability space $(\Omega,\mathcal{F},P,\mathbb{F}=(\mathcal{F}_t)_{0\leq t\leq T})$, we consider the zero-sum optimal stopping games 
$$\overline C:=\inf_{\Rho}\sup_{\tau\in\T}\E[U(\Rho(\tau),\tau)]\quad\text{and}\quad\underline C:=\sup_{\Tau}\inf_{\rho\in\T}\E[U(\Rho(\tau),\tau)]$$
in continuous time, where $U(s,t)$ is $\mathcal{F}_{s\vee t}$-measurable, $\T$ is the set of stopping times, and $\Rho,\Tau: \T\mapsto\T$ satisfy certain non-anticipativity conditions. In order to avoid the technical difficulties stemming from the verification of path regularity of some related processes (whether they are right continuous and have left limits), we work within the general framework of optimal stopping developed in \cite{Ko1,Ko2,Ko3}.  We convert the problems into a corresponding Dynkin game, and show that $\overline C=\underline C=V$, where $V$ is the value of the Dynkin game. This result extends \cite{Zhou1} to the continuous-time case and can be viewed as an application of the results in \cite{Ko3}, which weakens the usual path regularity assumptions on the reward processes.

It is worth noting that in \cite{Zhou1} two different types of non-anticipativity conditions are imposed for $\overline C$ and $\underline C$ respectively, for otherwise it can be the case that $\overline C\neq\underline C$. Now in the continuous-time case, we still have this inequality in general (see Remark 2.1).
But by assuming $U$ is right continuous along stopping times in the sense of expectation as in \cite{Ko1}, we are able to show that there is no essential difference between the two types of non-anticipativity conditions. 

The rest of the paper is organized as follows. In the next section, we introduce the setup and the main result. In Section 3, we give the proof of the main result. In section 4, we briefly discuss about the existence of optimal stopping strategies. 
\section{The setup and the main result}
Let $(\Omega,\mathcal{F},\mathbb{F},P)$ be a filtrated probability space, where $\mathbb{F}=(\mathcal{F}_t)_{0\leq t\leq T}$ is the filtration satisfying the usual conditions with $T\in(0,\infty)$ the time horizon in continuous time. Let $\T_t$ and $\T_{t+}$ be the set of $\mathbb{F}$-stopping times taking values in $[t,T]$ and $(t,T]$ respectively, $t\in[0,T)$. Denote $\T_T:=\T_{T+}:=\{T\}$ and $\T:=\T_0$. We shall often omit \lq\lq a.s.\rq\rq\ when a property holds outside a $P$-null set. Recall the definition of admissible families of random variables, e.g., in \cite{Ko1}.
\begin{definition}
A family $\{X(\sigma),\ \sigma\in\T\}$ is admissible if for all $\sigma\in\T$, $X(\sigma)$ is a bounded $\mathcal{F}_\sigma$-measurable random variable, and for all $\sigma_1, \sigma_2\in\T$, $X(\sigma_1)=X(\sigma_2)$ on $\{\sigma_1=\sigma_2\}$.
\end{definition}
\begin{definition}
A family $\{Y(\rho,\tau),\ \rho,\tau\in\T\}$ is biadmissible if for all $\rho,\tau\in\T$, $Y(\rho,\tau)$ is an $\mathcal{F}_{\rho\vee\tau}$-measurable bounded random variable, and for all $\rho_1, \rho_2,\tau_1,\tau_2\in\T$, $Y(\rho_1,\tau_1)=Y(\rho_2,\tau_2)$ on $\{\rho_1=\rho_2\}\cap\{\tau_1=\tau_2\}$.
\end{definition}
Let us also recall the two types of stopping strategies defined in \cite{Zhou1}.
\begin{definition}
$\Rho$ is a stopping strategy of Type I (resp. II), if $\Rho:\ \T\mapsto\T$ satisfies the \lq\lq non-anticipativity\rq\rq\ condition of Type I (resp. II), i.e., for any $\sigma_1,\sigma_2\in\T$, it holds a.s. that
\begin{equation}\label{e1}
\text{either}\quad\Rho(\sigma_1)=\Rho(\sigma_2)\leq\text{(resp. $<$) }\sigma_1\wedge\sigma_2\quad\text{or}\quad \Rho(\sigma_1)\wedge\Rho(\sigma_2)>\text{(resp. $\geq$) }\sigma_1\wedge\sigma_2.
\end{equation}
Denote by $\bT^i$ (resp. $\bT^{ii}$) the set of stopping strategies of Type I (resp. II).
\end{definition}
Below is an interesting property for the non-anticipative stopping strategies of Type I (but not Type II).
\begin{proposition}\label{p1}
For any $\Rho\in\bT^i$,
$$\Rho(\Rho(T))=\Rho(T).$$
\end{proposition}
\begin{proof}
Since 
$$\Rho(\Rho(T))\wedge\Rho(T)\leq\Rho(T)=\Rho(T)\wedge T,$$
by \eqref{e1} we have that
$$\Rho(\Rho(T))=\Rho(T)\leq\Rho(T)\wedge T.$$
\end{proof}

Let $\{U(\rho,\tau),\ \rho,\tau\in\T\}$ be an biadmissible family. Consider the optimal stopping games
\begin{equation}\notag
\overline A:=\inf_{\Rho\in\bT^i}\sup_{\tau\in\T}\E[U(\Rho(\tau),\tau)]\quad\text{and}\quad\underline A:=\sup_{\Tau\in\bT^i}\inf_{\rho\in\T}\E[U(\rho,\Tau(\rho))].
\end{equation}
and
\begin{equation}\notag
\overline B:=\inf_{\Rho\in\bT^{ii}}\sup_{\tau\in\T}\E[U(\Rho(\tau),\tau)]\quad\text{and}\quad\underline B:=\sup_{\Tau\in\bT^{ii}}\inf_{\rho\in\T}\E[U(\rho,\Tau(\rho))].
\end{equation}
We shall convert the problems into a corresponding Dynkin game. In order to do so, let us introduce two families of random variables that will represent the payoffs in the Dynkin game.
\begin{equation}\label{e3}
V^1(\tau):=\essinf_{\rho\in\T_\tau}\E_\tau[U(\rho,\tau)],\quad\tau\in\T
\end{equation}
and
\begin{equation}\label{e13}
V^2(\rho):=\esssup_{\tau\in\T_\rho}\E_\rho[U(\rho,\tau)],\quad\rho\in\T,
\end{equation}
where $\E_t[\cdot]=\E[\cdot|\mathcal{F}_t]$. Observe that
\begin{equation}\notag
V^1(\sigma)\leq U(\sigma,\sigma)\leq V^2(\sigma),\quad\sigma\in\T.
\end{equation}
Define the corresponding Dynkin game as follows:
\begin{eqnarray}
\overline V&:=&\inf_{\rho\in\T}\sup_{\tau\in\T}\E\left[V^1(\tau) 1_{\{\tau
\leq\rho\}}+V^2(\rho) 1_{\{\tau>\rho\}}\right],\label{e2}\\
\underline V&:=&\sup_{\tau\in\T}\inf_{\rho\in\T}\E\left[V^1(\tau) 1_{\{\tau\leq\rho\}}+V^2(\rho) 1_{\{\tau>\rho\}}\right].\label{e4}
\end{eqnarray}

Recall the (uniform) right continuity in expectation along stopping times defined in, e.g., \cite{Ko1}.
\begin{definition}
An admissible family $\{X(\sigma),\ \sigma\in\T\}$ is said to be right continuous along stopping times in expectation (RCE) if for any $\sigma\in\T$ and any $(\sigma_n)_n\subset\T$ with $\sigma_n\searrow\sigma$, one has 
$$\E[X(\sigma)]=\lim_{n\rightarrow\infty}\E[X(\sigma_n)].$$
\end{definition}
\begin{definition}
A biadmissible family $\{Y(\rho,\tau),\ \rho,\tau\in\T\}$ is said to be uniformly right continuous along stopping times in expectation (URCE) if  $\sup_{\rho,\tau \in \T}\E[Y(\rho,\tau)]<\infty$ and if for any $\rho,\tau\in\T$ and any $(\rho_n)_n,(\tau_n)_n\subset\T$ with $\rho_n\searrow\rho$ and $\tau_n\searrow\tau$, one has 
$$\lim_{n\rightarrow\infty}\sup_{\rho\in\T}\left|\E\left[Y(\rho,\tau)-Y(\rho,\tau_n)\right]\right|=0\quad\text{and}\quad\lim_{n\rightarrow\infty}\sup_{\tau\in\T}\left|\E\left[Y(\rho,\tau)-Y(\rho_n,\tau)\right]\right|=0.$$
\end{definition}
Below is the main result of this paper. 
\begin{theorem}\label{t1}
Assume the biadmissible family $\{U(\rho,\tau),\ \rho,\tau\in\T\}$ is URCE. We have that
$$\overline A=\underline A=\overline B=\underline B=\overline V=\underline V.$$
\end{theorem}

\begin{remark}
Without the right continuity assumption of $U$, it may fail that $\overline A=\underline A$ or $\overline B=\underline B$, even for some natural choices of $U$. For example. let $U(s,t)=|f(s)-f(t)|$, where
\begin{displaymath}
   f(t) = \left\{
     \begin{array}{lr}
       0\ , & 0\leq t\leq T/2,\\
       1\ , & T/2< t\leq T.
     \end{array}
   \right.
\end{displaymath}
Then the problems related to $\overline A,\underline A,\overline B,\underline B$ become deterministic.

Let us first show that $\overline A=1$. Take $\Rho\in\bT^i$. If $\Rho(T)\leq T/2$, then by taking $\tau=T$ we have that $\overline A=1$. Otherwise $\Rho(T)>T/2$, and we take $\tau=T/2$; Then by the non-anticipativity condition \eqref{e1}, we have that $\Rho(T/2)\wedge\Rho(T)>(T/2)\wedge T=T/2$, which implies $\overline A=1$. Next, consider $\underline A$. For any $\Tau\in\bT^i$, by \prref{p1} $\Tau(\Tau(T))=\Tau(T)$. Then letting $\rho=\Tau(T)$ we have that $\underline A=0$. Therefore, $\overline A\neq \underline A$.

Now by taking $\Rho(\tau)=\tau$ we have that $\overline B=0$.  Let us consider $\underline B$. Let $\Tau\in\bT^{ii}$ defined as
\begin{displaymath}
   \Tau(\rho) = \left\{
     \begin{array}{lr}
       T\ , & 0\leq \rho\leq T/2,\\
       T/2\ , & T/2<\rho\leq T.
     \end{array}
   \right.
\end{displaymath}
Then for any $\rho\in\T$, $U(\rho,\Tau(\rho))=1$ and thus $\underline B=1$. Therefore, $\overline B\neq \underline B$.
\end{remark}

\subsection{A sufficient condition for  $U$ to be URCE}
Let $W:[0,T]\times[0,T]\times\mathbb{R}\times\mathbb{R}\mapsto\mathbb{R}$ be $\mathcal{B}([0,T])\otimes\mathcal{B}([0,T])\otimes\mathcal{B}(\mathbb{R})\otimes\mathcal{B}(\mathbb{R})$-measurable. Assume that $W$ satisfies the Lipschitz condition, i.e., there exists some $L\in(0,\infty)$ such that
$$|W(s_1,t_1,x_1,y_1)-W(s_2,t_2,x_2,y_2)|\leq L(|s_1-s_2|+|t_1-t_2|+|x_1-x_2|+|y_1-y_2|).$$
Let $f=(f_t)_{0\leq t\leq T}$ and $g=(g_t)_{0\leq t\leq T}$ be two bounded and right continuous $\mathbb{F}$-progressively measurable processes.
\begin{proposition}
The family $\{U(\rho,\tau):=W(\rho,\tau,f_\rho,g_\tau),\ \rho,\tau\in\T\}$ is biadmissible and URCE. 
\end{proposition}
\begin{proof}
The biadmissibility is easy to check.  Let us check $U$ satisfies URCE: For any $\rho,\tau\in\T$ and any $(\tau_n)_n\subset\T$ with $\tau_n\searrow\tau$, we have that 
$$\lim_{n\rightarrow\infty}\sup_{\rho\in\T}\left|\E\left[U(\rho,\tau)-U(\rho,\tau_n)\right]\right|\leq L\lim_{n\rightarrow\infty}\E\left[|\tau-\tau_n|+|f_\tau-f_{\tau_n}|\right]=0.$$
\end{proof}

\subsection{An application}
 Let
$U(t,s)=\mathcal{U}(f_t-g_s),$
where $\mathcal{U}: \mathbb{R}\mapsto\mathbb{R}$ is a utility function, and $f$ and $g$ are two right continuous progressively measurable process. Consider
$$\overline{\mathbb{B}}:=\sup_{\Rho\in\bT^{ii}}\inf_{\tau\in\T}\E[U(\Rho(\tau),\tau)].$$
This problem can be interpreted as the one in which an investor longs an American option $f$ and shorts an American option $g$, and the goal is to choose an optimal stopping strategy to maximize the utility according to the stopping behavior of the holder of $g$. Here we assume that the maturities of $f$ and $g$ are the same (i.e., $T$). This is without loss of generality. Indeed for instance, if the maturity of $f$ is $\hat t<T$, then we can define $f(t)=f(\hat t)$ for $t\in(\hat t,T]$.

\section{Proof of \thref{t1}}
We will only prove that $\overline A=\overline V=\underline V$, and the proof we provide in this section also works for $\underline A,\overline B$ and $\underline B$. Throughout this section, we assume that the biadmissible family $\{U(\rho,\tau),\ \rho,\tau\in\T\}$ is URCE. 
\begin{lemma}\label{l1}
$\overline V=\underline V.$
\end{lemma}
\begin{proof}
The argument below (2.2) in \cite{Ko1} shows that $\{V^1(\tau),\ \tau\in\T\}$ and $\{V^2(\rho),\ \rho\in\T\}$ are admissible. By \cite[Theorem 2.2]{Ko1},  $V^1$ and $V^2$ are RCE because U is assumed to be URCE. Then by \cite[Theorem 3.6]{Ko3} we have that $\overline V=\underline V$.
\end{proof}

\begin{proposition}
The values of $\overline V$ and $\underline V$ do not change if we replace $\{\tau\leq\rho\}$ and $\{\tau>\rho\}$ in \eqref{e2} and \eqref{e4} with $\{\tau<\rho\}$ and $\{\tau\geq\rho\}$ respectively.
\end{proposition}
\begin{proof}
Define
$$\overline\cV:=\inf_{\rho\in\T}\sup_{\tau\in\T}\E\left[V^1(\tau) 1_{\{\tau
<\rho\}}+V^2(\rho) 1_{\{\tau\geq\rho\}}\right],$$
$$\underline\cV:=\sup_{\tau\in\T}\inf_{\rho\in\T}\E\left[V^1(\tau) 1_{\{\tau<\rho\}}+V^2(\rho) 1_{\{\tau\geq\rho\}}\right].$$
Since $U$ is URCE, the families $V^1$ and $V^2$ are RCE by \cite[Theorem 2.2]{Ko1}, and hence \cite[Theorem 3.6]{Ko3} can be applied to the Dynkin game with value functions $\overline V$ and $\underline V$, as well as the Dynkin game with value functions $-\overline\cV$ and $-\underline\cV$. Now by construction (see \cite{Ko3}), the families $\mathcal{J}$ and $\mathcal{J}'$ associated with the Dynkin game $(-\overline\cV,-\underline\cV)$ satisfy $\mathcal{J}=J'$ and $\mathcal{J}'=J$,
where $J$ and $J'$ are two nonnegative supermartingale families satisfying
$$J(\sigma)=\esssup_{\tau\in\T_\sigma}\E_\sigma[J'(\tau)+V^1(\tau)],$$
$$J'(\sigma)=\esssup_{\rho\in\T_\sigma}\E_\sigma[J(\rho)-V^2(\rho)];$$
It follows that
$$\overline\cV=\underline\cV=J(0)-J'(0)=\overline V=\underline V.$$
\end{proof}
\begin{lemma}\label{l2}
For any $\eps>0$ and $\tau\in\T$, there exists $\rho_\tau\in\T_{\tau+}$, such that
$$\E|\E_\tau[U(\rho_\tau,\tau)]-V^1(\tau)|<\eps.$$
A similar result holds for $V^2$.
\end{lemma}
\begin{proof}
First let us show that
\begin{equation}\label{e333}
V^1(\tau)=\essinf_{\rho\in\T_{\tau+}}\E_\tau[U(\rho,\tau)].
\end{equation}
Fix $\tau\in\mathcal{T}$. Define the value function for each stopping time $S$ by 
$$\mathbb{U}^\tau(S) := \essinf_{\rho\in\mathcal{T}_S}\E_S[U(\rho,\tau)].$$
The family $\{\mathbb{U}^\tau(S),\ S\in\mathcal{T}\}$ is clearly RCE. By \cite[Proposition 1.15]{Ko2}, we have that
$$\mathbb{U}^\tau(S) = \essinf_{\rho\in\mathcal{T}_{S+}}\E_S[U(\rho,\tau)].$$
Taking $S=\tau$ in the above, and since $V^1(\tau)=\mathbb{U}^\tau(\tau)$, we obtain \eqref{e333}.

Next, fix $\tau\in\T$. Since the family $\{\E_\tau[U(\rho,\tau)]: \rho\in\T_{\tau+}\}$ is closed under pairwise minimization, by, e.g., \cite[Theorem A.3]{KS2}, there exists $(\rho_n)\in\T_{\tau+}$ such that
$$\lim_{n\rightarrow\infty}\E_\tau[U(\rho_n,\tau)]=\essinf_{\rho\in\T_{\tau+}}\E_\tau[U(\rho,\tau)]=V^1(\tau).$$
Since $U$ and $V^1(\tau)$ are bounded, we have that
$$\lim_{n\rightarrow\infty}\E|\E_\tau[U(\rho_n,\tau)]-V^1(\tau)|=0,$$
which implies the result.
\end{proof}

\begin{lemma}\label{l3}
$\overline A\leq\overline V.$
\end{lemma}
\begin{proof}
Take $\eps>0$. Let $\rho_\eps\in\T$ be an $\eps$-optimizer of $\overline V$, i.e.,
$$\sup_{\tau\in\T}\E\left[1_{\{\rho_\eps\leq\tau\}}V^2(\rho_\eps)+1_{\{\rho_\eps
>\tau\}}V^1(\tau)\right]<\overline V+\eps.$$
For any $\tau\in\T$, by \leref{l2} there exists $\rho_\eps^1(\tau)\in\T_{\tau+}$ such that
$$\E|\E_\tau[U(\rho_\eps^1(\tau),\tau)-V^1(\tau)|<\eps.$$
Define $\Rho_\eps$ as
\begin{equation}\label{e5}
\Rho_\eps(\tau):=\rho_\eps 1_{\{\tau\geq\rho_\eps\}}+\rho_\eps^1(\tau) 1_{\{\tau<\rho_\eps\}},\quad\tau\in\T.
\end{equation}

Let us show that $\Rho_\eps$ is in $\bT^i$. First, for any $\tau\in\T$, $\Rho_\eps(\tau)\in\T$ since for any $t\in[0,T]$,
\begin{eqnarray}
\notag \{\Rho_\eps(\tau)\leq t\}&=&\left(\{\tau\geq\rho_\eps\}\cap\{\rho_\eps\leq t\}\right)\cup\left(\{\tau<\rho_\eps\}\cap\{\rho_\eps^1(\tau)\leq t\}\right)\\
\notag &=&\left(\{\tau\geq\rho_\eps\}\cap\{\rho_\eps\leq t\}\right)\cup\left(\{\tau<\rho_\eps\}\cap\{\tau\leq t\}\cap\{\rho_\eps^1(\tau)\leq t\}\right)\in\mathcal{F}_t.
\end{eqnarray}
Then let us show that $\Rho_\eps$ satisfies the non-anticipativity condition of Type I in \eqref{e1}. Take $\tau_1,\tau_2\in\T$. On $\{\Rho_\eps(\tau_1)\leq\tau_1\wedge\tau_2\leq\tau_1\}$, if $\tau_1<\rho_\eps\leq T$, then $\Rho_\eps(\tau_1)=\rho_\eps^1(\tau_1)>\tau_1$, contradiction. Hence $\tau_1\geq \rho_\eps$, and thus $\Rho_\eps(\tau_1)=\rho_\eps\leq\tau_1\wedge\tau_2\leq\tau_2$, which implies $\Rho_\eps(\tau_2)=\rho_\eps=\Rho_\eps(\tau_1)\leq\tau_1\wedge\tau_2$. On $\{\Rho_\eps(\tau_1)>\tau_1\wedge\tau_2\}$, if $\Rho_\eps(\tau_2)\leq\tau_1\wedge\tau_2$ then we can use the previous argument to get that $\Rho_\eps(\tau_1)=\Rho_\eps(\tau_2)
\leq\tau_1\wedge\tau_2$ which is a contradiction, and thus $\Rho_\eps(\tau_2)>\tau_1\wedge\tau_2$.

We have that
\begin{eqnarray}
\notag \overline A&\leq& \sup_{\tau\in\T}\E[U(\Rho_\eps(\tau),\tau)]\\
\notag &=& \sup_{\tau\in\T}\E\left[U(\rho_\eps,\tau) 1_{\{\rho_\eps\leq\tau\}}+U(\rho_\eps^1(\tau),\tau) 1_{\{\rho_\eps>\tau\}}\right]\\
\notag &=& \sup_{\tau\in\T}\E\left[1_{\{\rho_\eps
\leq\tau\}}\E_{\rho_\eps}[U(\rho_\eps,\tau)] +1_{\{\rho_\eps>\tau\}}\E_\tau[U(\rho_\eps^1(\tau),\tau) \right]\\
\notag&\leq&\sup_{\tau\in\T}\E\left[1_{\{\rho_\eps
\leq\tau\}}V^2(\rho_\eps)+1_{\{\rho_\eps
>\tau\}}V^1(\tau)\right]+\eps\\
\notag&
\leq&\overline V+2\eps.
\end{eqnarray}
\end{proof}

\begin{remark}
Once we show \thref{t1}, we can see that $\Rho_\eps\in\bT^i\subset\bT^{ii}$ defined in \eqref{e5} is a $2\eps$-optimizer for $\overline A$ and $\overline B$.
\end{remark}

\begin{lemma}\label{l4}
$\overline A\geq\underline V.$
\end{lemma}
\begin{proof}
Fix $\eps>0$. Let $\tau_\eps\in\T$ be an $\eps$-optimizer for $\underline V$. For any $\rho\in\T$, by \leref{l2} there exists  $\tau_\eps^2(\rho)\in\T_{\rho+}$ such that
$$\E|\E_\tau[U(\rho,\tau_\eps^2(\rho))-V^2(\rho)|<\eps$$
For any $\Rho\in\bT^{i}$, define $\tau_\Rho$ as
$$\tau_\Rho:=\tau_\eps 1_{\{\tau_\eps\leq\Rho(\tau_\eps)\}}+\tau_\eps^2(\Rho(\tau_\eps)) 1_{\{\tau_\eps>\Rho(\tau_\eps)\}}.$$
Using a similar argument as in the proof of \leref{l3}, we can show that $\tau_\Rho$ is a stopping time. 

Since $\bT^i\subset\bT^{ii}$, and also in order to let our proof also fit for $\overline B$, we shall only use the non-anticipativity condition of Type II for $\Rho$ in \eqref{e1}, although $\Rho\in\bT^i$. By \eqref{e1} w.r.t. Type II, it holds a.s. that
$$\text{either}\quad\Rho(\tau_\Rho)=\Rho(\tau_\eps)<\tau_\eps\wedge\tau_\Rho\quad\text{or}\quad\Rho(\tau_\Rho)\wedge\Rho(\tau_\eps)\geq\tau_\eps\wedge\tau_\Rho.$$
Therefore, 
$$\text{on}\quad\{\tau_\eps\leq\Rho(\tau_\eps)\},\quad\text{we have}\quad\Rho(\tau_\Rho)\geq\tau_\eps\wedge\tau_\Rho=\tau_\eps=\tau_\Rho,$$
and
$$\text{on}\quad\{\tau_\eps>\Rho(\tau_\eps)\},\quad\text{we have}\quad\tau_\Rho=\tau^2_\varepsilon(\Rho(\tau_\eps))>\Rho(\tau_\eps)\implies\Rho(\tau_\eps)<\tau_\Rho\wedge\tau_\eps\implies\Rho(\tau_\eps)=\Rho(\tau_\Rho).$$

Now we have that
\begin{eqnarray}
\notag\sup_{\tau\in\T}\E[U(\Rho(\tau),\tau))]&\geq&\E[U(\Rho(\tau_\Rho),\tau_\Rho)]\\
\notag&=&\E\left[U(\Rho(\tau_\Rho),\tau_\Rho) 1_{\{\tau_\eps\leq\Rho(\tau_\eps)\}}+U(\Rho(\tau_\Rho),\tau_\Rho) 1_{\{\tau_\eps>\Rho(\tau_\eps)\}}\right]\\
\notag&=&\E\left[U(\Rho(\tau_\Rho),\tau_\eps) 1_{\{\tau_\eps
\leq\Rho(\tau_\eps)\}}+U(\Rho(\tau_\eps),\tau_\eps^2(\Rho(\tau_\eps))) 1_{\{\tau_\eps>\Rho(\tau_\eps)\}}\right]\\
\notag&=&\E\left[1_{\{\tau_\eps
\leq\Rho(\tau_\eps)\}}\E_{\tau_\eps}[U(\Rho(\tau_\Rho),\tau_\eps)] +1_{\{\tau_\eps>\Rho(\tau_\eps)\}}\E_{\Rho(\tau_\eps)}[U(\Rho(\tau_\eps),\tau_\eps^2(\Rho(\tau_\eps)))]\right]\\
\notag&\geq&\E\left[1_{\{\tau_\eps
\leq\Rho(\tau_\eps)\}}V^1(\tau_\eps)+1_{\{\tau_\eps>\Rho(\tau_\eps)\}}V^2(\Rho(\tau_\eps))\right]-\eps\\
\notag&\geq&\inf_{\rho\in\T}\E\left[1_{\{\tau_\eps
\leq\rho\}}V^1(\tau_\eps)+1_{\{\tau_\eps>\rho\}}V^2(\rho)\right]-\eps\\
\notag&\geq&\underline V-2\eps,
\end{eqnarray}
where the fifth inequality follows from the definition of $V^1$ in \eqref{e3} and the fact that $\Rho(\tau_\Rho)\geq\tau_\eps$ on $\{\tau_\eps\leq\Rho(\tau_\eps)\}$. By the arbitrariness of $\Rho\in\bT^i$ and $\eps$, the conclusion follows.
\end{proof}

\begin{proof}[\textbf{Proof of \thref{t1}}]
This follows from Lemmas \ref{l1}, \ref{l3} and \ref{l4}.
\end{proof}

\section{Existence of optimal stopping strategies}
If we impose a strong left continuity assumptions on $U$ in addition (see e.g. \cite{Ko2,Ko1,Ko3}), we would get the existence of the optimal stopping strategies for $\overline B$ and $\underline B$. For example let us consider $\overline B$. Indeed, the left and right continuity of $U$ would imply the required left and right continuity of $V^1$ and $V^2$, as well as the existence of an optimal stopping time $\rho_0^1(\tau)\in\T_\tau$ for $V^1(\tau)$. The continuity of $V^1$ and $V^2$ would further imply the existence of an optimal stopping time $\rho_0$ for $\overline V$ (see \cite{Ko3}). Then define
$$\Rho_0(\tau):=\rho_0 1_{\{\tau\geq\rho_0\}}+\rho_0^1(\tau) 1_{\{\tau<\rho_0\}},\quad\tau\in\T.$$
Following the proof of \leref{l3}, one can show that $\Rho_0\in\bT^{ii}$ is optimal for $\overline B$. One should note that in this case $\Rho_0$ may not be in $\bT^i$ as opposed to $\Rho_\eps$ define in \eqref{e5}, this is because here it is possible that $\rho_0(\tau)=\tau$ on $\{\tau<T\}$. 

On the other hand, the existence of optimal stopping strategies for $\overline A$ and $\underline A$ may fail in general even if $U$ is quite regular. For example, let $U(s,t)=|s-t|$. By taking $\Rho(\tau)=\tau$ we have that $\overline B=0$ which is equal to $\overline A$ by Theorem~\ref{t1}. Now assume there exists some optimal $\hat\Rho\in\bT^i$ for $\overline A$. That is,
$$\sup_\tau|\hat\Rho(\tau)-\tau|=\overline A=0.$$
Then we have that $\Rho(\tau)=\tau$ for any $\tau\in[0,T]$, which contradicts with the non-anticipativity condition of Type I  by letting $\sigma_1\neq \sigma_2$ in \eqref{e1}.

\bibliographystyle{siam}
\bibliography{ref}

\end{document}